\documentclass[12pt]{amsart}

\usepackage[export]{adjustbox}
\usepackage{amsrefs}
\usepackage{amssymb}
\usepackage{fancyhdr}
\usepackage{bbm}
\usepackage{xcolor}
\usepackage{hyperref}
\hypersetup{
  colorlinks=true,
  linkcolor=blue,
  citecolor=red
}

\newtheorem{theorem}{Theorem}
\newtheorem*{theorem*}{Theorem}
\newtheorem{lemma}[theorem]{Lemma}
\newtheorem{proposition}[theorem]{Proposition}
\newtheorem{claim}[theorem]{Claim}
\newtheorem{corollary}[theorem]{Corollary}

\newtheorem{maintheorem}{Theorem}

\theoremstyle{definition}
\newtheorem{definition}[theorem]{Definition}
\newtheorem*{definition*}{Definition}
\newtheorem*{lemma*}{Lemma}
\usepackage{graphicx}
\usepackage{pstricks, enumerate, pst-node, pst-text, pst-plot}

\numberwithin{equation}{section}
\numberwithin{theorem}{section}

\newcommand{\N}{\mathbb{N}}

\newcommand{\eps}{\varepsilon}

\usepackage{xparse}
\DeclareDocumentCommand\Pr{ m g }{%
    {   \IfNoValueTF {#2}
      {\mathbb{P}\left[{#1}\right]}
      {\mathbb{P}\left[{#1}\middle\vert{#2}\right]}%
    }
}
\DeclareDocumentCommand\E{ m g }{%
    {   \IfNoValueTF {#2}
      {\mathbb{E}\left[{#1}\right]}
      {\mathbb{E}\left[{#1}\middle\vert{#2}\right]}%
    }
}
\def\cc{\curvearrowright}

\begin{document}

\title[]{Strong amenability and the infinite conjugacy class property}

\author[]{Joshua Frisch}
\author[]{Omer Tamuz}
\author[]{Pooya Vahidi Ferdowsi}
\address{California Institute of Technology}
\thanks{This work was supported by  a grant from the Simons Foundation (\#419427,  Omer Tamuz)}

%    Information for third author

%\thanks{}
\date{\today}

\begin{abstract}
  A group is said to be strongly amenable if each of its proximal
  topological actions has a fixed point. We show that a finitely
  generated group is strongly amenable if and only if it is virtually
  nilpotent. More generally, a countable discrete group is strongly
  amenable if and only if none of its quotients have the infinite
  conjugacy class property.
\end{abstract}

\maketitle
%\tableofcontents

\section{Introduction}

Let $G \cc X$ be a continuous action of a countable discrete group on a compact Hausdorff space. This action is said to be {\em proximal} if for any $x,y \in X$ there exists a net $\{g_i\}$ in $G$ such that $\lim_i g_i x=\lim_i g_i y$. $G$ is said to be {\em strongly amenable} if every such proximal action of $G$ has a fixed point. Glasner introduced these notions in~\cite{glasner1976proximal}, and proved a number of results: he showed that every virtually nilpotent group is strongly amenable, and that non-amenable groups are not strongly amenable. He also gave some examples of amenable groups that are not strongly amenable.\footnote{Glasner attributes one of these examples to Furstenberg.} Since then, a number of papers have studied strong amenability~\cites{glasner1983proximal, glasner2002minimal, melleray2015polish, hartman2016thompson, dai2017universal}, but none have made significant progress on relating it to other group properties.

Our main result is a characterization of strongly amenable groups. Recall that a group has the infinite conjugacy class property (ICC) if each of its non-trivial elements has an infinite conjugacy class.
\begin{maintheorem}
 \label{thm:main}
   A countable discrete group is strongly amenable if and only if it
   has no ICC quotients. In particular, a finitely generated group is
   strongly amenable if and only if it is virtually nilpotent.
\end{maintheorem}
For example, this implies that the group $S_\infty$ of finite permutations of $\N$ is not strongly amenable. Likewise, the alternating subgroup of $S_\infty$ is not strongly amenable, as is every infinite simple group. 

Groups that have no ICC quotients are also known as {\em
  hyper-FC-central}~\cite{mclain1956remarks} or {\em
  hyper-FC}~\cite{duguid1960class}. By Theorem~\ref{thm:main}, these are exactly the strongly amenable groups. This implies that subgroups of
strongly amenable countable discrete groups are again strongly
amenable. A finitely generated group is hyper-FC if and only if it is
virtually
nilpotent~\cite{mclain1956remarks,duguid1956fc}\footnote{For a
  self-contained proof see~\cite{frisch2018virtually}.}, from
which the second part of the theorem follows.

The case of groups with no ICC quotients is a straightforward consequence of Glasner's work. 
To prove that groups with ICC quotients are not strongly amenable, we consider an ICC group $G$, and a certain class of symbolic dynamical systems for $G$. Using a topological genericity argument, we show that in this class there is a proximal action without a fixed point.

In light of Theorem~\ref{thm:main}, it is natural to ask whether, in some larger class of topological groups, a group is strongly amenable if and only if each of its quotients has a non-trivial element with a compact (or perhaps precompact) conjugacy class.

\subsection*{The hyper-FC-subgroup, the universal minimal proximal action, and the group von Neumann algebra}

One can define the {\em FC-series} of a group $G$ as
$$1\leq F_1\leq F_2 \leq \cdots \leq F_\alpha \leq \cdots \leq G,$$
where $F_{\alpha+1}/F_\alpha$ is the normal subgroup of $G/F_\alpha$ consisting of the elements of the finite conjugacy classes, and $F_\beta = \cup_{\alpha<\beta} F_\alpha$ for $\beta$ a limit ordinal~\cites{haimo1953FC, mclain1956remarks}. The group $\cup_\alpha F_\alpha$ is called the {\em hyper-FC-subgroup} of $G$.

In ~\cite[Section II.4]{glasner1976proximal} Glasner defines the universal minimal proximal action of a group $G$; this is the unique minimal proximal action of $G$ which has every minimal proximal action as a factor. We denote this action by $G\cc \partial_p G$. In Proposition~\ref{prop:faithful} we show that every ICC group has a minimal proximal faithful action. On the other hand, the proof of Proposition~\ref{prop:non-icc} shows that the hyper-FC-subgroup of $G$ acts trivially on $\partial_p G$. Combining these gives us:
\begin{corollary}
\label{cor:kernel}
  For a countable discrete group $G$, $\ker(G\cc \partial_p G)$ is equal to the hyper-FC-subgroup of $G$.
\end{corollary}

Glasner also defines the universal minimal {\em strongly proximal} action of a group $G$, which is the unique minimal strongly proximal  action of $G$ which has every minimal strongly proximal action as a factor\footnote{A topological action $G \cc X$ is {\em strongly proximal} if for each Borel probability measure $\mu$ on $X$ there exists a net $\{g_i\}$ such that $\lim_i g_i\mu$ is a point mass.}. We denote this action by $G \cc \partial_s G$. Furman~\cite[Proposition 7]{furman2003minimal} shows that the kernel of $G \cc \partial_s G$ is the amenable radical of $G$.

It is known that the group von Neumann algebra of a group $G$ has a unique tracial state iff $G$ is ICC, and we show that ICC groups are precisely the groups with faithful universal minimal proximal actions. We thus have the following dynamical characterization of the unique trace property of the group von Neumann algebra:
\begin{corollary}
\label{cor:von Neumann unique trace}
  For a countable discrete group $G$, the following are equivalent:
  \begin{enumerate}
      \item The group von Neumann algebra of $G$ has a unique tracial state.
      \item $G\cc \partial_p G$ is faithful.
  \end{enumerate}
\end{corollary}

Analogously, it has been recently shown by Breuillard, Kalantar, Kennedy, and Ozawa~\cite[Corollary 4.3]{breuillard2017c} that the reduced $\mathrm{C}^*$-algebra of $G$ has a unique tracial state if and only if $G \cc \partial_s G$ is faithful.

Following this analogy raises an interesting question. We know from~\cite[Theorem 1.5]{kalantar2017boundaries} that simplicity of the reduced $\mathrm{C}^*$-algebra of a group $G$ is equivalent to the freeness of $G\cc \partial_s G$. We also know from~\cite[Corollary 4.3]{breuillard2017c} that unique trace property of the reduced $\mathrm{C}^*$-algebra of a group $G$ is equivalent to faithfulness of $G\cc \partial_s G$. On the other hand, for group von Neumann algebras of discrete groups, simplicity and the unique trace property are equivalent. So it is natural to ask whether freeness of $G\cc \partial_p G$ is equivalent to its faithfulness.

\subsection*{Acknowledgments}
We would like to thank Benjamin Weiss and Andrew Zucker for correcting mistakes in earlier drafts of this paper, and to likewise thank an anonymous referee for many corrections and suggestions. We would also like to thank Yair Hartman and Mehrdad Kalantar for drawing our attention to the relation of our results to the unique trace property of group von Neumann algebras.

\section{Overview of the proof of Theorem~\ref{thm:main}}
That a group with no ICC quotients is strongly amenable  follows immediately from 
from the following proposition.
\begin{proposition}
\label{prop:non-icc}
  Let $G$ be a countable discrete group that acts faithfully, minimally and proximally on a compact Hausdorff space $X$. Then each non-trivial element of $G$ has an infinite conjugacy class.
\end{proposition}
\begin{proof}
  Let $g$ be a non-trivial element of $G$. Assume by contradiction that $g$ has a finite conjugacy class. Let $H$ be the centralizer of $g$, so that $H$ has finite index in $G$. By~\cite[Lemma 3.2]{glasner1976proximal} $H$ acts proximally and minimally on $X$. Since $g$ is in the center of $H$, it acts trivially on $X$, by~\cite[Lemma 3.3]{glasner1976proximal}. This contradicts the assumption that the action is faithful.
\end{proof}

Thus, to prove Theorem~\ref{thm:main}, we consider any $G$ that is ICC, and prove that it has a proximal action that does not have a fixed point. This is without loss of generality, since if $G$ has a proximal action without a fixed point, then so does any group that has $G$ as a quotient.

\subsection{Existence by genericity}
Our general strategy for the proof of Theorem~\ref{thm:main} is to consider a certain space $\mathcal{S}$ of non-trivial actions of $G$. We show that this space includes a proximal action without a fixed point by showing that, in fact, a {\em generic} action in this space is minimal and proximal. Genericity here is in the Baire category sense.

To define the space $\mathcal{S}$, let $A$ be a finite alphabet of size at least 2. The {\em full shift} $A^G$, equipped with the product topology, is a space on which $G$ acts continuously by left translations. Enumerate elements of $G = \{g_1, g_2, \ldots \}$ and endow $A^G$ with the metric $d(\cdot,\cdot)$ given by $d(s,t)=1/k$ where $k=\inf\{n \,:\, s(g_n)\neq t(g_n)\}$. An element of $A^G$ is called a {\em configuration}.

The closed, $G$-invariant non-empty subsets of $A^G$ are called {\em shifts}. The space of shifts is endowed with the subspace topology of the Hausdorff topology (or Fell topology) on the closed subsets of $A^G$. This topology is also metrizable: take, for example, the metric that assigns to a pair of shifts $S,T \subseteq A^G$ the distance $1/(n+1)$, where $n$ is the largest index such that  $S$ and $T$ agree on $\{g_1,\ldots,g_n\}$; by agreement on a finite $X\subseteq G$ we mean that the restriction of the configurations in $S$ to $X$ is equal to the restrictions of the configurations in $T$ to $X$. Note that for any shift $S\subseteq A^G$, the sets of the form $\{T\subseteq A^G \:|\: T \text{ agrees with } S \text{ on } X\}$ for different finite subsets $X\subseteq G$ form a basis of the neighborhoods for $S$.

We define the space $\mathcal{S}$ to be the closure, in the space of shifts, of the strongly irreducible shifts, with the $|A|$-many trivial (i.e., singleton) shifts removed. Strongly irreducible shifts  are defined as follows:
\begin{definition}
  A shift $S\subseteq A^G$ is said to be strongly irreducible if there exists a finite $X\subseteq G$ including the identity such that for any two subsets $E_1,E_2 \subseteq G$ with $E_1 X \cap E_2 X = \emptyset$ and any two configurations $s_1,s_2\in S$ there is a configuration $s\in S$ such that $s$ restricted to $E_1$ equals $s_1$ restricted to $E_1$, and $s$ restricted  to $E_2$ equals $s_2$ restricted to $E_2$. 
\end{definition}

To show that the proximal actions are generic in $\mathcal{S}$, we define $\eps$-proximal actions; proximal actions will be the actions which are $\eps$-proximal for each $\eps>0$.
\begin{definition}
  An action $G \cc X$ on a compact metric space with metric $d(\cdot,\cdot)$ is $\eps$-proximal if
  for all $x,y \in X$ there exists a $g \in G$ such that
  $d(g x,g y) < \varepsilon$.
\end{definition}

To show that minimal actions are generic in $\mathcal{S}$, we similarly define the notion of $\eps$-minimality.
\begin{definition}
  An action $G \cc X$ on a compact metric space with metric $d(\cdot,\cdot)$ is $\eps$-minimal if for all $x,y\in X$ there exists a $g\in G$ such that $d(g x, y)<\varepsilon$.
\end{definition}

A subset of a topological space is generic (in the Baire category sense) if it contains a dense $G_\delta$. To prove our main result, we show that the proximal actions are a dense $G_\delta$ in $\mathcal S$. The proof of density is the main challenge of this paper, while proving that this subset is a $G_\delta$ is straightforward.
\begin{claim}
  \label{clm:G-delta}
    The set of $\eps$-proximal shifts is an open set in $\mathcal S$. Thus the set of proximal shifts is a $G_\delta$ set in $\mathcal S$.
    
    Similarly, the set of $\eps$-minimal shifts is an open set in $\mathcal{S}$. Thus the set of minimal shifts is a $G_\delta$ set in $\mathcal{S}$.
\end{claim}

The Baire Category Theorem guarantees that for well behaved spaces (such as our locally compact space  $\mathcal S$), a countable intersection of dense open sets is dense. Thus, to prove that the proximal shifts are dense in the closure of the strongly irreducible shifts, it suffices to show that the $\eps$-proximal shifts are dense in $\mathcal S$ for each $\eps$. That is, fixing $\eps$, we must show that for each strongly irreducible shift $S \subseteq A^G$ and each finite subset $X \subseteq G$ there exists a strongly irreducible shift $S'$ that agrees with $S$ on $X$, and is $\eps$-proximal.

To this end, we construct a class of shifts of $\{0,1\}^G$ (which we denote by $2^G$) which are $\eps$-proximal. Furthermore, for these shifts $\eps$-proximality  is witnessed by a particular configuration around the origin: one having a 1 at the origin, and zeros close to it. For a finite symmetric subset $X \subset G$ and $g,h \in G$ we say that $g$ and $h$ are {\em $X$-apart} if $g^{-1}h \not \in X$. 
\begin{definition}
  Let $X$ be a finite symmetric subset of $G$. A non-trivial shift $S \subset 2^G$ is an {\em $X$-witness} shift if
  \begin{enumerate}
  \item For each $s \in S$, $s(a)=1$ and $s(b)=1$ implies
    that $a$ and $b$ are $X$-apart.
  \item For each $s,t \in S$ there exists an $a \in G$ such that
    $s(a)=t(a)=1$.
  \end{enumerate}
\end{definition}
The construction of $X$-witness shifts in
Propositions~\ref{prop:r-witness for finite sets}
and~\ref{prop:r-witness si} contains the main technical effort of this
paper.

\subsection{A toy example}
\label{sec:informal}  
To give the reader some intuition and explain the role of ICC in the
construction of $X$-witness shifts, we now explain how to construct a
single configuration in $2^G$ with an $X$-witness {\em orbit}, and
show that such configurations do not exist for groups that are not
ICC. Note that the closure of this orbit is not necessarily an
$X$-witness shift; the construction of $X$-witness shifts requires
more work and a somewhat different approach, which we pursue later, in
the formal proofs.

Given a configuration $u \in 2^G$, we denote by
$G u = \{g u\,:\, g \in G\}$ the $G$-orbit of $u$. Given a finite
symmetric $X \subset G$, we say that a configuration $u\in 2^G$ is an
{\em $X$-witness configuration} if
\begin{enumerate}
\item For each $s \in G u$, $s(a)=1$ and $s(b)=1$ implies
  that $a$ and $b$ are $X$-apart.
\item For each $s,t \in G u$ there exists an $a \in G$ such that
  $s(a)=t(a)=1$.
\end{enumerate}
We now informally explain that when $G$ is ICC then for every such $X$
there exist $X$-witness configurations, and that when $G$ is not ICC
then there is a finite symmetric $X \subset G$ for which there are no such
configurations.

Suppose first that $G$ is not ICC. Then there cannot exist an
$X$-witness shift for every $X$. To see this, suppose that $g\in G$ is
an element with finitely many conjugates, and let $X$ be a finite
symmetric subset of $G$ that contains all the conjugates of
$g$. Assume towards a contradiction that there exists an $X$-witness
configuration $u$. So, by the second property of $X$-witness
configurations, there exists an $a\in G$, such that $[gu](a)=u(a)=1$,
which means $u(g^{-1}a)=u(a)=1$. Now, by the first property of $u$, we
need to have that $g^{-1}a$ and $a$ are $X$-apart, which means
$(g^{-1}a)^{-1}a = a^{-1}ga \notin X$. This is a contradiction, since
we let $X$ contain all the conjugates of $g$.

Consider now the case that $G$ is ICC. Given a finite symmetric $X$,
we choose a random configuration $u \in 2^G$ as follows. Assign to
each element of $G$ an independent uniform random variable in
$[0,1]$. Let $V_a$ be the random variable corresponding to $a\in
G$. For each $a\in G$, let $u(a) = 1$ iff $V_a > V_{a x}$ for all
$x\in X\setminus \{e\}$; i.e., $u(a)=1$ if $V_a$ is maximal in its
$X$-neighborhood. Note that if $g$ and $h$ are $X^2$-apart then the
event $u(g)=1$ and the event $u(h)=1$ are independent.

We claim that $u$ is, with probability one, an $X$-witness shift.  By
construction $u$ almost surely satisfies the first property: if
$s = g^{-1} u$ and $s(a)=s(b)=1$ then $u(g a)=u(g b)=1$, hence $g a$
and $g b$ are $X$-apart, and so $a$ and $b$ are $X$-apart. To satisfy
the second property, it must hold that for every $g \neq h \in G$
there is some $a \in G$ such that $u(ga) = u(ha)=1$. By the ICC
property, we can choose an $a \in G$ to make $g a$ and $h a$
arbitrarily far apart, as this corresponds to finding an $a$ such that
$a^{-1} (h^{-1}g) a$ is large. For such a choice of $a$, the events
$u(g a)=1$ and $u(h a)=1$ are independent, and, since we have
infinitely many such $a$'s that we can use, with probability
one at least one of them will give us the desired result.

\subsection{Constructing $X$-witness shifts}
We now return to the construction of $X$-witness shifts, which are the main tool in our proof of Theorem~\ref{thm:main}. The first step is to construct a single configuration which is an $X$-witness in a large finite set.
\begin{proposition}
\label{prop:r-witness for finite sets}
  Let $G$ be an ICC group. For each finite symmetric $X \subset G$ there exists an $s \in 2^G$ and a finite symmetric $Y \supset X$ such that 
  \begin{enumerate}
      \item For every $a,b \in G$, if $s(a)=s(b)=1$ then $a$ and $b$ are $X$-apart.
      \item For every $g,h \in Y^{100}$ there exists some $a \in Y$ such that $s(g a) = s(h a) = 1$.
  \end{enumerate}
\end{proposition}
The proof of this proposition---along with
Proposition~\ref{prop:r-witness si} below---contains the main
technical effort of this paper. The proof elaborates on the ideas of
the informal construction of \S\ref{sec:informal}:  we choose the configuration $s$ at random, and then show that it has the desired properties with positive probability. This stage crucially uses the assumption that the group is ICC, which translates to independence of some events that arise in the analysis of this random choice. This is the only step in the proof of Theorem~\ref{thm:main} in which we use the ICC property of $G$.

We use the configuration constructed in Proposition~\ref{prop:r-witness for finite sets} to construct $X$-witness shifts. These shifts will additionally (and importantly) be strongly irreducible.
\begin{proposition}
\label{prop:r-witness si}
Let $G$ be a group for which, for each finite symmetric $X \subset G$, there exists a configuration that satisfies the conditions of Proposition~\ref{prop:r-witness for finite sets}. Then for each such $X$ there also exists a strongly irreducible $X$-witness shift.
\end{proposition}
The combination of Propositions~\ref{prop:r-witness for finite sets} and~\ref{prop:r-witness si} immediately yields the following. 
\begin{proposition}
\label{prop:r-witness si amenable icc-rf}
  Let $G$ be an ICC group. Then for each finite symmetric $X \subset G$ there exists a strongly irreducible $X$-witness shift $S \subset 2^G$.
\end{proposition}

\subsection{$\eps$-proximal shifts}
Finally, we use these strongly irreducible $X$-witness shifts to construct approximations to a given strongly irreducible shift $S$ that are both $\eps$-proximal and $\eps$-minimal.
\begin{proposition}
\label{prop:eps-proximal approx}
  Let $G$ be a group for which there exists, for each finite symmetric $X \subset G$, a strongly
  irreducible $X$-witness shift.  Let $T \subseteq A^G$ be a strongly
  irreducible shift. Then for each $\eps$ and finite $X\subset G$ there exists a strongly irreducible shift $T' \subseteq 2^G$ that is $\eps$-proximal, $\eps$-minimal, and agrees with $T$ on $X$.
\end{proposition}

An immediate consequence of Proposition~\ref{prop:r-witness si amenable icc-rf}, Proposition~\ref{prop:eps-proximal approx}, and  Claim~\ref{clm:G-delta} is the following.
\begin{proposition}
\label{prop:dense-G_delta}
 Let $G$ be an ICC group. Then there is a dense $G_\delta$ set in $\mathcal S$ for which the action   $G \cc S$ is  minimal and proximal.
 \end{proposition}

In the next proposition we show that this result can be strengthened to show that a generic shift is additionally faithful.
 \begin{proposition}
 \label{prop:faithful}
 Let $G$ be an ICC group. Then there is a dense $G_\delta$ set in $\mathcal S$ for which the action   $G \cc S$ is {\em faithful}, minimal, and proximal.
 \end{proposition}

Given all this, the proof of our main theorem follows easily.
\begin{proof}[Proof of Theorem~\ref{thm:main}]
 That groups with no ICC quotients are strongly amenable follows immediately from Proposition~\ref{prop:non-icc}. Let $G$ be ICC. By Proposition~\ref{prop:dense-G_delta} the proximal minimal shifts are a dense $G_\delta$ in $\mathcal S$, and in particular exist, since $\mathcal S$ is non-empty (e.g., the full shift $A^G$ is strongly irreducible and non-constant). Since there are no trivial shifts in $\mathcal S$, and since non-trivial minimal shifts have no fixed points, we have proved that $G$ is not strongly amenable. 
\end{proof}

\section{Proofs}

\subsection{Proof of Proposition \ref{prop:r-witness for finite sets}}
  Let $G$ be an ICC group, and let $X$ be a finite, symmetric subset of $G$. We choose a random configuration $u \in 2^G$ as follows. 
  Assign to each element of $G$ an independent uniform random variable in $[0,1]$. Let $V_a$ be the random variable corresponding to $a\in G$. For each $a\in G$, let $u(a) = 1$ iff $V_a > V_{ax}$ for all $x\in X\setminus \{e\}$. That is, let $u(a)=1$ if $V_a > V_b$ whenever $a^{-1}b\in X$ and $b\neq a$. The following claim is an immediate consequence of the definition of $u$.
  \begin{claim}
  \label{clm:independence}
     If $a_1,\ldots, a_n$ are $X^2$-apart\footnote{Recall that given a finite symmetric subset $X \subset G$ and $g, h \in G$, we say that $g$ and $h$ are $X$-apart if $g^{-1}h \not \in X$.} for $a_i\in G$, then $\{u(a_i)=1\}$ are independent events.
  \end{claim}

  Clearly, for all values of the random configuration, $u(a)=u(b)=1$ implies that $a^{-1}b\notin X$ for all $a,b\in G$, which means that $a$ and $b$ are $X$-apart. So the random configuration $u$ almost surely satisfies the first part of the proposition. It thus remains to find a finite symmetric subset $Y \supset X$ such that, with positive probability for the random configuration $u$, for each $g,h \in Y^{100}$ there exists some $a \in Y$ such that $u(g a) = u(h a) = 1$.
  
  The next lemma claims that there exists a subset $Y$ with certain useful properties. We use this lemma to prove our proposition, and then prove the lemma.
  \begin{lemma}
  \label{lem:Y}
    There exists a $Y \supset X$ with the following properties.
    \begin{enumerate}
        \item $$|Y|^{200} (1-|X|^{-2})^{|Y|/(20|X^2|+5)} < 1.$$
        \item For each $g,h \in G$ there exists a subset $Y_{g,h} \subseteq Y$ with the following properties.
        \begin{enumerate}
            \item $|Y_{g,h}| \geq |Y|/(20|X^2|+5)$.
            \item For $y\in Y_{g,h}$, $g y$ and $h y$ are $X^2$-apart.
            \item For $y_1\neq y_2\in Y_{g,h}$, $w_1$ and $w_2$ are $X^2$-apart for any $w_1\in\{g y_1, h y_1\}$ and $w_2\in\{g y_2, h y_2\}$.
        \end{enumerate}
    \end{enumerate}
  \end{lemma}
  
  For $c,d,y \in G$, let $E_c$ be the event that $u(c)=1$, and let $E_{c,d}^y = E_{c y} \cap E_{d y}$. Now fix $g,h \in G$. 
  \begin{enumerate}

    \item By the second property of $Y_{g,h}$, $g y$ and $h y$ are $X^2$-apart for any $y\in Y_{g,h}$. Hence $E_{gy}$ and $E_{hy}$ are independent, by Claim~\ref{clm:independence}.
    
    \item $\Pr{E_c} = 1/|X|$ for all $c\in G$.
    
    \item Combining the previous two results: $\Pr{E_{g,h}^y} = |X|^{-2}$ for all $y\in Y_{g,h}$. So $\Pr{\neg E_{g,h}^y} = 1-|X|^{-2}$.
    
    \item $E_{g,h}^y$ are independent events for different values of $y\in Y_{g,h}$. This is because (I) $g y$ and $h y$ are $X^2$-apart for any $y\in Y_{g,h}$, and (II) $w_1$ and $w_2$ are $X^2$-apart for any $y_1\neq y_2\in Y_{g,h}$, $w_1\in\{g y_1, h y_1\}$, and $w_2\in\{g y_2, h y_2\}$, which means $\{ E_{g y}, E_{h y} \:|\: y\in Y_{g,h}\}$ are independent events. And finally, since $E_{g,h}^y = E_{gy} \cap E_{hy}$, we get that $E_{g,h}^y$ are independent events for $y\in Y_{g,h}$.

    \item We say that the pair $(g, h)$ {\em fails} if $E_{g,h}^y$ does not happen for any $y\in Y_{g,h}$. So, by the previous two results,
    \begin{align*}
        \Pr{(g, h) \text{ fails}} & = \Pr{E_{g,h}^y \text{ for no } y\in Y_{g,h}} \\
        & = (1-|X|^{-2})^{|Y_{g,h}|} \\
        & \leq (1-|X|^{-2})^{|Y|/(20|X^2|+5)},
    \end{align*}
    where the last inequality follows from the first property of $Y_{g,h}$.
  \end{enumerate}

  By the last inequality, union bound, and the first property of $Y$:
  \begin{align*}
    \Pr{(g,h) \text{ fails for some } g,h\in Y^{100} } & \leq |Y^{100}|^2(1-|X|^{-2})^{|Y|/(20|X^2|+5)} \\
    & \leq |Y|^{200}(1-|X|^{-2})^{|Y|/(20|X^2|+5)} < 1.
  \end{align*}
  So, there is at least one configuration, say $s$, for which no $(g,h)$ fails for $g,h\in Y^{100}$. Therefore, for all $g,h\in Y^{100}$ there is an $a\in Y$ such that $s(ga) = s(ha) = 1$. So this $s$ satisfies the second part of the proposition, which concludes the proof of Proposition ~\ref{prop:r-witness for finite sets}, except the proof of Lemma~\ref{lem:Y}, to which we turn now.

  \begin{proof}[Proof of Lemma~\ref{lem:Y}]
  We call an element $g\in G$ {\em switching} if for all non-identity $x \in X^2$ we have $g^{-1} x g\notin X^2$.
  \begin{claim}
    There exists at least one switching element $g_s\in G$.
  \end{claim}
  \begin{proof}
  Let $C_x$ be the centralizer of $x$ for each $x\in X^2$. Then there are finitely many cosets of $C_x$, say $g^x_1 C_x, \ldots, g^x_{n_x} C_x$, such that $g^{-1}xg\in X^2$ only if $g\in g^x_i C_x$ for some $i\in \{1,\ldots,n_x\}$.  So, non-switching elements are in the union of finitely many cosets of subgroups with infinite index, i.e.\ $g$ is non-switching only if $g\in g^x_i C_x$ for some $x\in X^2$ and some $i\in\{1,\ldots, n_x\}$. Since $G$ is ICC, each $C_x$ has infinite index in $G$. By ~\cite[Lemma 4.1]{neumann1954groups} a finite collection of cosets of infinite index does not cover the whole group $G$, so, there is at least one switching element in $G$.
  \end{proof}
  
  Let $g_s$ be a switching element. We can choose an arbitrarily large finite subset $Y_1\subseteq G$ which includes the identity and such that $Y_1\cap Y_1 g_s=\emptyset$. Choose such a $Y_1$ that is large enough so  that 
  $$(5 |Y_1|)^{200}(1-|X|^{-2})^{2|Y_1|/(20|X^2|+5)} < 1 \;\text{ and }\; |Y_1|\geq |X|$$ 
  and let $Y = (Y_1\cup Y_1 g_s) \cup (Y_1\cup Y_1 g_s)^{-1} \cup X$. Note that $Y$ is symmetric and  $5|Y_1| \geq |Y| \geq 2 |Y_1|$, which implies $$|Y|^{200} (1-|X|^{-2})^{|Y|/(20|X^2|+5)} < 1.$$ This establishes the first property of $Y$.
  
  Fix $g,h\in G$ with $g\neq h$. We say $y\in G$ is {\em distancing} for the pair $(g,h)$ if $gy$ and $hy$ are $X^2$-apart.
  \begin{claim}
  If $y\in G$ is not distancing for $(g,h)$ then $y g_s$ is distancing for $(g,h)$.  
  \end{claim}
  \begin{proof}
   Since $y$ is not distancing for $(g,h)$, $(g y)^{-1}(h y) = y^{-1}g^{-1}h y \in X^2$. By the definition of a switching element $g_s^{-1} [(g y)^{-1}(h y)] g_s = (g y g_s)^{-1}(h y g_s) \notin X^2$, which means that $y g_s$ is distancing for $(g,h)$.
  \end{proof}
    By this observation, if $y_1\in Y_1$ is not distancing for $(g,h)$ then $y_1 g_s\in Y_1 g_s$ is distancing for $(g,h)$. So at least half of the elements in $Y_1 \cup Y_1 g_s$ are distancing for $(g,h)$ and thus at least one fifth of the elements in $Y$ are distancing for $(g,h)$. Let $Y'_{g,h}$ be the collection of elements in $Y$ that are distancing for $(g,h)$. We just saw that $|Y'_{g,h}|\geq |Y|/5$.
  
  Now define a graph on $Y'_{g,h}$ by connecting $y_1\neq y_2\in Y'_{g,h}$ if $w_1$ and $w_2$ are not $X^2$-apart for some $w_1\in\{g y_1, h y_1\}, w_2\in\{g y_2, h y_2\}$. Call this graph $G'_{g,h}$. Note that the degree of each $y\in Y'_{g,h}$ in $G'_{g,h}$ is at most $4|X^2|$. So, we can find an independent set of size at least $|Y'_{g,h}|/(4|X^2|+1) \geq |Y|/(20|X^2|+5)$ in $G'_{g,h}$. Call this independent set $Y_{g,h}$.
  
  \begin{claim}
    $|Y_{g,h}| \geq |Y|/(20|X^2|+5)$, for $y\in Y_{g,h}$, $g y$ and $h y$ are $X^2$-apart, and for $y_1\neq y_2\in Y_{g,h}$, we have that $w_1$ and $w_2$ are $X^2$-apart for $w_1\in\{g y_1, h y_1\}, w_2\in\{g y_2, h y_2\}$.
  \end{claim}
  \begin{proof}
     The bound on the size of $Y_{g,h}$ is established in the previous paragraph. Since $Y_{g,h}\subseteq Y'_{g,h}$ and all elements of $Y'_{g,h}$ are distancing for $(g,h)$, the second property holds. The third property follows from independence of $Y_{g,h}$ in $G'_{g,h}$.
  \end{proof}
  This establishes the three properties of $Y_{g,h}$, and thus concludes the proof of the lemma.
  \end{proof}

\subsection{Saturdated packings}
In this section we prove some general claims regarding {\em saturated packings} (see, e.g., \cite{toth1993packing}).

  \begin{definition}
    Let $Z_1,\ldots,Z_n$ be distinct non-empty finite subsets of $G$. A {\em $\{Z_1,\ldots,Z_n\}$-packing} is a $p\in \{Z_1,\ldots,Z_n,\emptyset\}^G$ with $h\:p(h)\cap g\:p(g)=\emptyset$ for all $g\neq h\in G$; note that $h\:p(h)$ and $g\:p(g)$ are each a translate, by $h$ and $g$ respectively, of some element of $\{Z_1,\ldots,Z_n,\emptyset\}$. When $p(g) \neq \emptyset$, we call the translate $g\:p(g)$ a {\em block}.
  \end{definition}
    By an abuse of notation, we use the term $Z$-packing instead of $\{Z\}$-packing when we have only one subset.

  \begin{definition}
    A  $\{Z_1,\ldots,Z_n\}$-packing $p$ is {\em saturated} if there is no $\{Z_1,\ldots,Z_n\}$-packing $p'\neq p$ such that $p(g)\neq \emptyset$ implies $p'(g)=p(g)$.
    
    We say that $p$ is a saturation of $q$ if $p$ is saturated and $q(g)\neq\emptyset$ implies $p(g)=q(g)$.
  \end{definition}
  Saturated packings are packings to which one cannot add any blocks. Note, however, that it may be possible to add more blocks  by first removing some. Note also that by Zorn's Lemma there exists for each $\{Z_1,\ldots,Z_n\}$-packing $q$ a $\{Z_1,\ldots,Z_n\}$-packing $p$ that saturates it.

  The following claim shows the existence of strongly irreducible saturated packings, which will be useful in the construction of strongly irreducible $X$-witness shifts.
  A similar claim with a similar proof appears in~\cite[Lemma 2.2]{frisch2017symbolic}. 

  Given two distinct non-empty finite subsets $Z_1$ and $Z_2$ of $G$, we denote by $\pi \colon \{Z_1, Z_2, \emptyset\}^G \to \{Z_1, \emptyset\}^G$ the map 
  $$
    (\pi(p))(g) = 
    \begin{cases}
      Z_1 & \text{if } p(g) = Z_1\\
      \emptyset & \text{otherwise.}
    \end{cases}
  $$
  That is, $\pi$ transforms a $\{Z_1,Z_2\}$-packing into a $Z_1$-packing by removing all the $Z_2$-blocks.
  \begin{claim}
  \label{clm:saturated-packings}
    Let $Z_1$ and $Z_2$ be two distinct non-empty finite subsets of $G$. Let $P$ be the collection of all saturated $\{Z_1,Z_2\}$-packings $p$ such that  $\pi(p)$ is a saturated $Z_1$-packing.
    Then $P$ is a non-empty strongly irreducible shift.
  \end{claim}
  \begin{proof}
    The proof of the fact that $P$ is a non-empty shift is standard. It thus remains to be shown that it is strongly irreducible.
    
    Let $X = (Z_1\cup Z_2) \cup (Z_1\cup Z_2)^{-1}$. Let $F_1, F_2\subset G$ be any two subsets of $G$ that are $X^{14}$-apart. To prove the claim, it suffices to show that for any $p_1,p_2 \in P$ there is a $p \in P$ that agrees with $p_1$ on $F_1$ and with $p_2$ on $F_2$. We know that $F_1 X^6$ and $F_2 X^6$ are disjoint, and furthermore, if $a_1 \in F_1 X^6$ and $a_2 \in F_2 X^6$ then the blocks $a_1 p_1(a_1)$ and $a_2 p_2(a_2)$ are disjoint.
    
    Let $q_1 = \pi(p_1)$ and $q_2 = \pi(p_2)$. We know that $q_1$ and $q_2$ are saturated $Z_1$-packings. We also know that if $a_1\in F_1 X^6$ and $a_2\in F_2 X^6$ then $a_1 q_1(a_1)$ and $a_2 q_2(a_2)$ are disjoint. Thus there is a $Z_1$-packing, say $q'$, that is equal to $q_1$ on $F_1 X^6$ and to $q_2$ on $F_2 X^6$. Let $q$ be a $Z_1$-packing that is a saturation of $q'$. Fix $i\in\{1,2\}$ and $g\in F_i X^4$. We will show that $q_i(g) = q(g)$.
    \begin{itemize}
        \item If $q_i(g) = Z_1$, we know that $q'(g) = Z_1$, and hence $q(g) = Z_1$.
        \item If $q_i(g) = \emptyset$, since $q_i$ is a saturated $Z_1$-packing, there exists $a\in g Z_1 Z_1^{-1} \subseteq F_i X^6$ with $q_i(a) = Z_1$. So, $q'(a) = Z_1$, and hence $q(a) = Z_1$. Since $gZ_1\cap aZ_1\neq\emptyset$, this implies that $q(g) = \emptyset$.
    \end{itemize}
    So, $q$ is a saturated $Z_1$-packing that agrees with $q_1$ on $F_1 X^4$ and with $q_2$ on $F_2 X^4$.
    
    Since $q$ agrees with $q_1 = \pi(p_1)$ on $F_1 X^4$ and with $q_2 = \pi(p_2)$ on $F_2 X^4$, it is easy to see that $p'$, which is defined as follows, is a well-defined $\{Z_1,Z_2\}$-packing:
      $$
        p'(g) = 
        \begin{cases}
          p_1(g) & \text{if } g\in F_1 X^2\\
          p_2(g) & \text{if } g\in F_2 X^2\\
          q(g) & \text{otherwise.}
        \end{cases}
      $$
    So, by definition, $p'$ agrees with $p_1$ on $F_1 X^2$ and with $p_2$ on $F_2 X^2$. Furthermore, $\pi(p')=q$, since $\pi(p_i)$ agrees with $q$ on $F_i X^2$. 
    
    Let $p$ be a $\{Z_1,Z_2\}$-packing that is a saturation of $p'$. Since $\pi(p') = q$ is a saturated $Z_1$-packing, we have $\pi(p) = \pi(p') = q$. So, $p$ is a saturated $\{Z_1,Z_2\}$-packing where $\pi(p)$ is a saturated $Z_1$-packing, which means $p\in P$. To complete the proof, we just need to show that $p$ agrees with $p_1$ on $F_1$ and with $p_2$ on $F_2$. Fix $i\in\{1,2\}$ and $g\in F_i$. We will show that $p_i(g) = p(g)$.
    \begin{itemize}
        \item If $p_i(g)\in \{Z_1, Z_2\}$, we know that $p'(g) = p_i(g) \in \{Z_1,Z_2\}$, and hence $p(g) = p_i(g)$.
        \item If $p_i(g) = \emptyset$, since $p_i$ is a saturated $\{Z_1,Z_2\}$-packing, for any $j\in\{1,2\}$ there exist $\ell\in\{1,2\}$ and $a\in g Z_j Z_\ell^{-1} \subseteq F_i X^2$ with $p_i(a) = Z_\ell$. So, $p'(a) = Z_\ell$, and hence $p(a) = Z_\ell$. Since $g Z_j\cap a Z_\ell\neq\emptyset$, this implies that $p(g) = \emptyset$.
    \end{itemize}
  \end{proof}
  
  %In particular, for any non-empty finite subset $Z\subset G$, by setting $Z_1 = Z_2 = Z$, this claim implies that the set of all saturated $Z$-packings is a non-empty strongly irreducible shift.

\subsection{Proof of Proposition~\ref{prop:r-witness si}}

  We can now start the proof of proposition~\ref{prop:r-witness si}. Assume that $X$, a finite symmetric subset of $G$, is given. We now seek to construct a strongly irreducible $X$-witness shift $T$. Since $G$ satisfies proposition~\ref{prop:r-witness for finite sets}, we can  let $Y$ and $s$ be a finite symmetric subset of $G$ and a configuration on $G$ that satisfy the statement of proposition~\ref{prop:r-witness for finite sets} for $X\subseteq G$.

  Let $P$ be the strongly irreducible shift given by Claim~\ref{clm:saturated-packings} for $Z_1 = Y^{100}X$ and $Z_2 = YX$.
  
  Define $\psi: P\to 2^G$ by
  $$
  [\psi(p)](g) =
  \begin{cases}
    s(h^{-1} g) & \text{if } g\in h\:Y^{100} \text{ for some } h\in G\\
                & \quad \text{with } p(h) = Y^{100}X \\
    s(h^{-1} g) & \text{if } g\in h\:Y \text{ for some } h\in G\\
                & \quad \text{with } p(h) = YX \\
    0 & \text{ otherwise. }
  \end{cases}
  $$
   What $\psi$ does is produce a configuration  which is 0 outside of the $X$-interior\footnote{The $X$-interiors of $Y^{100}X$ and $YX$ are $Y^{100}$ and $Y$.} of blocks, and is equal to translates of $s|_{Y^{100}}$ and $s|_Y$ inside the interior of blocks.
  
  It is again easy to see that $\psi$ is continuous and equivariant, so $T=\psi(P)$ is a strongly irreducible shift. The following claim completes the proof of proposition~\ref{prop:r-witness si}.

  \begin{claim}
    $T$ is an $X$-witness shift.
  \end{claim}
  
  The claim follows immediately from the following two lemmas.
  The first of the lemmas is straightforward from our construction, while the second is less immediate.
  
  \begin{lemma}
    For all $t\in T$, the $1$'s in $t$ are $X$-apart.
  \end{lemma}
  \begin{proof}
    Let $t\in T$ and $a,b\in G$ with $t(a)=t(b)=1$. Since $t\in T$, there is a $p\in P$ with $\psi(p)=t$. By the definition of $\psi$, since $[\psi(p)](a)=[\psi(p)](b)=1$, we get that $a\in h\:p(h)$ and $b\in g\:p(g)$ for some $h,g\in G$. If $g=h$, i.e.\ $a$ and $b$ are in the same block of $p$, then $s(h^{-1}a)=t(a)=1$ and $s(h^{-1}b)=t(b)=1$. But, since $s$ satisfies proposition~\ref{prop:r-witness for finite sets} (in particular, the $1$'s in $s$ are $X$-apart), $h^{-1}a$ and $h^{-1}b$ are $X$-apart, which implies $a$ and $b$ are $X$-apart. If $g\neq h$, then $h\:p(h)$ and $g\:p(g)$ are disjoint, so the $X$-interior of $h\:p(h)$ and the $X$-interior of $g\:p(g)$ are $X$-apart. We also know that $a$ is in the $X$-interior of $h\:p(h)$ and $b$ is in the $X$-interior of $g\:p(g)$. Therefore, $a$ and $b$ are $X$-apart.
  \end{proof}
  
  \begin{lemma}
    For any $t_1, t_2\in T$ there is an $a\in G$ with $t_1(a)=t_2(a)=1$.
  \end{lemma}
  \begin{proof}
    We essentially prove this lemma by a series of reductions.
    
    Let $t_1, t_2\in T$. So there are $p_1, p_2\in P$ with $\psi(p_1)=t_1$ and $\psi(p_2)=t_2$. Pick an $a_1\in G$ with $p_1(a_1)=Y^{100}X$. This means that $a_1^{-1}p_1$ has a block of shape $Y^{100}X$ centered at the identity. Let $p_1'=a_1^{-1}p_1$, $p_2'=a_1^{-1}p_2$, and let $t_1'=\psi(p_1')$, $t_2'=\psi(p_2')$. So, $p_1'$ has a block of shape $Y^{100}X$ centered at the identity.
    
    Since $p_2'$ is saturated, we know there is an $a_2\in Y^4$ such that either (I) $a_2$ is in the $YX$-interior of a block of shape $Y^{100}X$ in $p_2'$, or (II) $a_2$ is the center of a block of shape $YX$ in $p_2'$. Let $p_1''=a_2^{-1}p_1'$, $p_2''=a_2^{-1}p_2'$, and let $t_1''=\psi(p_1'')$, $t_2''=\psi(p_2'')$. Observe that in $p_1''$ the identity is in the $YX$-interior of a block of shape $Y^{100}X$, say $e\in k_1 Y^{99}$ for some $k_1\in G$ with $p_1''(k_1)=Y^{100}X$. Moreover, in $p_2''$ the identity is either (I) in the $YX$-interior of a block of shape $Y^{100}X$ or (II) in the center of a block of shape $YX$.
    
    \begin{figure}
    \includegraphics[scale=0.6,center]{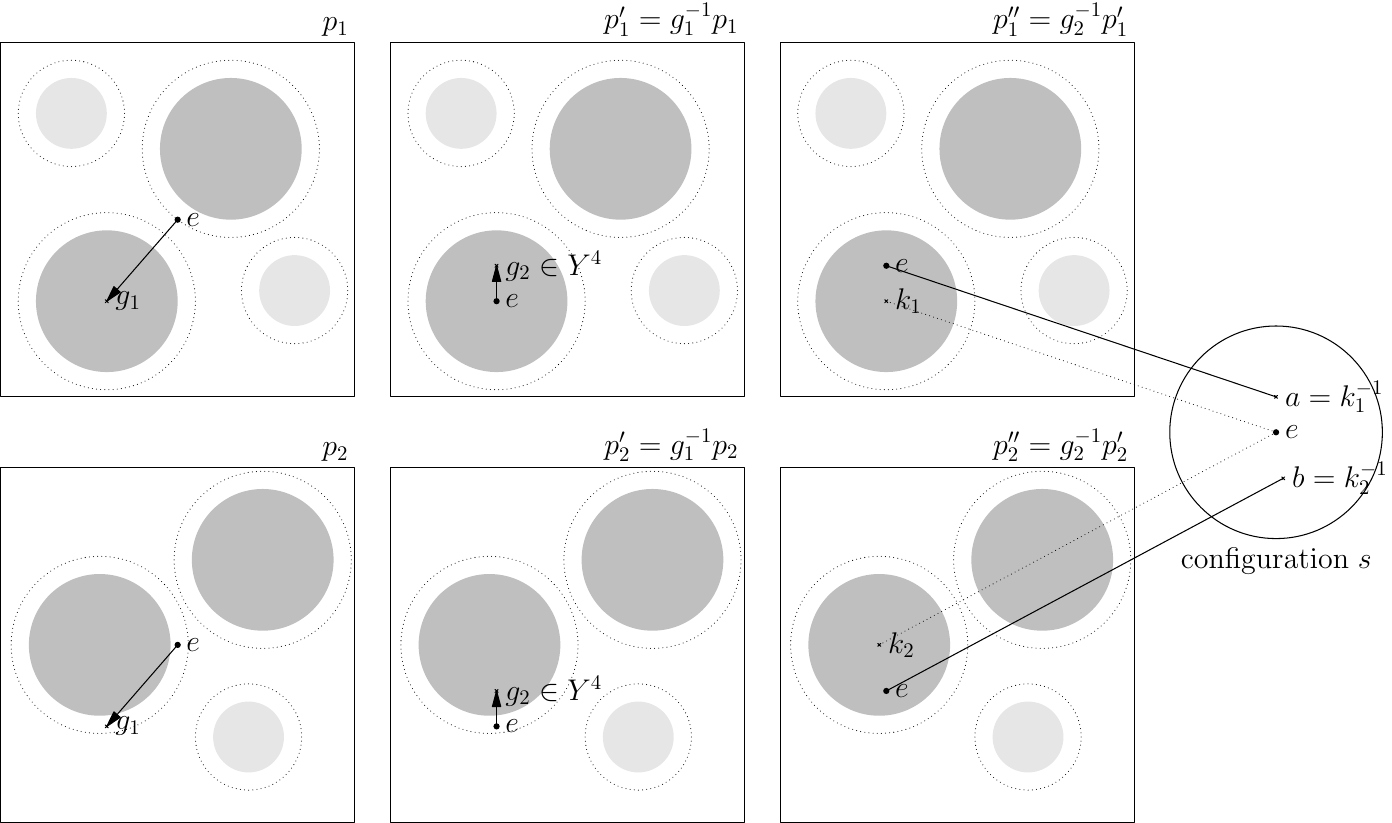}
      \caption{\label{fig:caseI}Case (I).}
    \end{figure}
    In case (I), since in $p_2''$ the identity is in the $YX$-interior of a block of shape $Y^{100}X$, $e\in k_2\:Y^{99}$ for some $k_2\in G$ with $p_2''(k_2)=Y^{100}X$. So $k_2^{-1}\in Y^{99}$. By the second part of proposition~\ref{prop:r-witness for finite sets} applied to $g=k_1^{-1}$ and $h=k_2^{-1}$, we know that there is an $a_3\in Y$ such that $s(k_1^{-1}a_3) = s(k_2^{-1} a_3) = 1$. So, by the definition of $\psi$, the fact that $k_1^{-1}a_3\in Y^{100}$, and the fact that $p_1''(k_1)=Y^{100}X$, we get $t_1''(a_3) = s(k_1^{-1} a_3) = 1$, and similarly, we get $t_2''(a_3) = s(k_2^{-1} a_3) =1$. Therefore, $t_1(a_1 a_2 a_3) = t_1''(a_3) = 1$ and $t_2(a_1 a_2 a_3) = t_2''(a_3) = 1$. Case (I) is schematically depicted in Figure~\ref{fig:caseI}.
    
    \begin{figure}
      \includegraphics[scale=0.6,center]{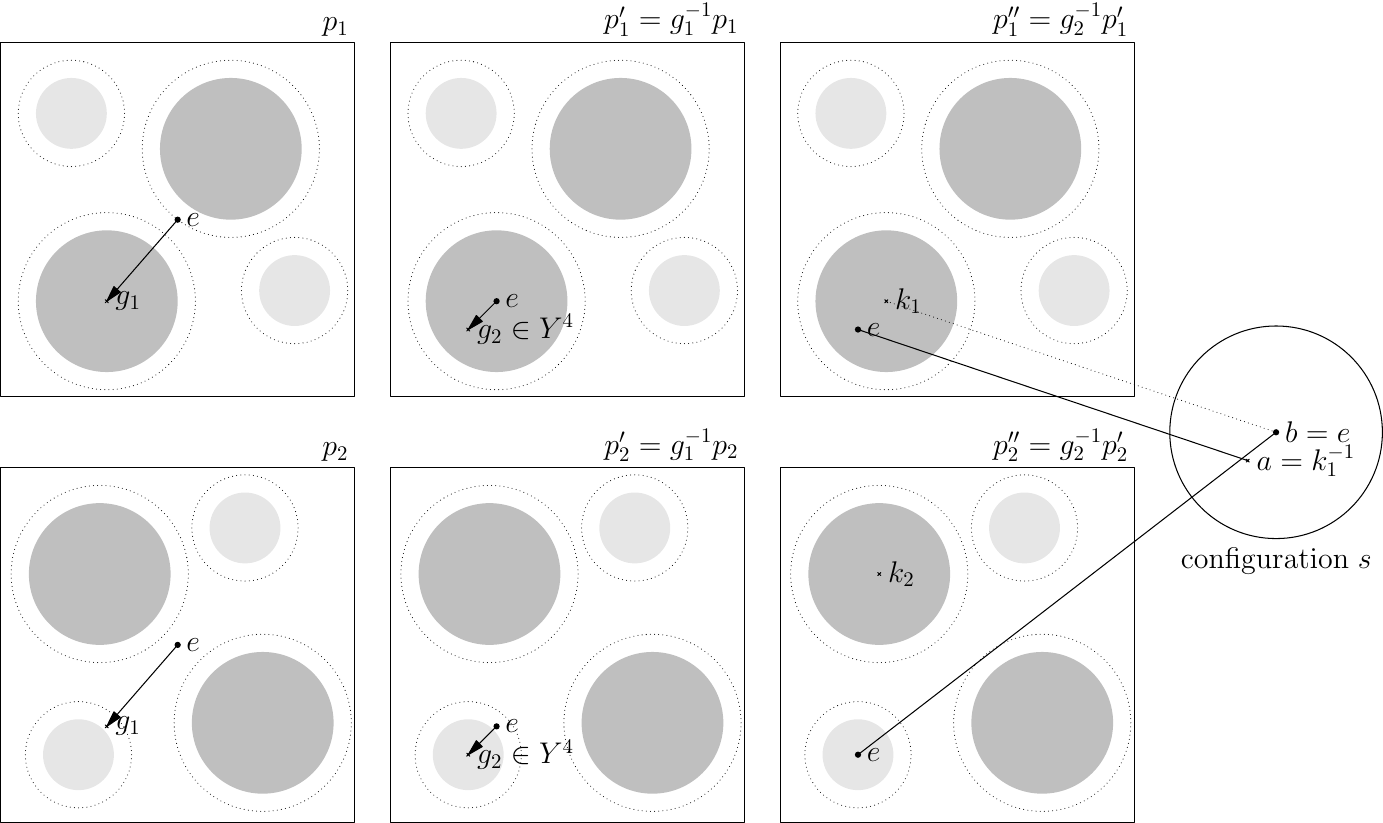}
      \caption{\label{fig:caseII}Case (II).}
    \end{figure}
    In case (II), $p_2''(e)=YX$. Again, if we apply the second part of proposition~\ref{prop:r-witness for finite sets} to $g=k_1^{-1}$ and $h=e$, we get that there is an $a_3\in Y$ such that $s(k_1^{-1}a_3) = s(a_3) = 1$. So, by the definition of $\psi$, the fact that $k_1^{-1}a_3\in Y^{100}$, and the fact that $p_1''(k_1)=Y^{100}X$, we get $t_1''(a_3) = s(k_1^{-1} a_3) = 1$. Also, by the definition of $\psi$, the fact that $a_3\in Y$, and the fact that $p_2''(e)=YX$, we get $t_2''(a_3)=s(a_3)=1$. Therefore, $t_1(a_1 a_2 a_3)=t_1''(a_3)=1$ and $t_2(a_1 a_2 a_3) = t_2''(a_3) = 1$. Case (II) is schematically depicted in Figure~\ref{fig:caseII}.
    
    In both cases we showed there is an $a\in G$ with $t_1(a)=t_2(a)=1$. This completes the proof.
  \end{proof}

\subsection{Proof of Proposition~\ref{prop:eps-proximal approx}}
  Fix $\epsilon>0$ and $X$ a finite symmetric subset of $G$. Without loss of generality, we may assume that $X$  includes the identity, and is large enough so that any two configurations that agree on $X$ have distance less than $\eps$.
  
  Since $T\subseteq A^G$ is a strongly irreducible shift, there is a finite symmetric $U\subseteq G$ including the identity such that for any two subsets $E_1,E_2 \subseteq G$ with $E_1 U \cap E_2 U = \emptyset$ and any two configurations $t_1,t_2\in T$ there is a configuration $t\in T$ such that $t$ restricted to $E_1$ equals $t_1$ restricted to $E_1$, and $t$ restricted  to $E_2$ equals $t_2$ restricted to $E_2$
  
  Given a shift $S \subseteq A^G$ and a finite $Y \subset G$, we call a map $p \colon Y \to A$ a {\em $Y$-pattern of $S$} if it is equal to  $s|_Y$, the restriction of some $s \in S$ to $Y$. In this case we say that $s$ contains the $Y$-pattern $p$.
  
  By strong irreducibility of $T$, we can find a $u\in T$ whose orbit $\{g u\,:\, g \in G\}$ contains all the $X$-patterns of $T$. Furthermore, since there are only finitely many $X$-patterns in $T$, there must be a finite $V \subset G$ (which we assume w.l.o.g.\ to be symmetric and contain the identity) such that $\{g u\,:\, g \in \text{ the } X\text{-interior of } V\}$ contains all the $X$-patterns of $T$. By making $V$ even larger, we can assume that $d(t,t') < \eps$ for any two configurations $t,t' \in T$ that agree on $V$, where $d(\cdot,\cdot)$ is the metric on $T$.  
  
   Let $Z=(VU^2X)(VU^2X)^{-1}$. By the assumption in the statement, there is a strongly irreducible $Z$-witness shift for $G$. Call this shift $S$.
   
  Now, define a continuous equivariant function $\phi: S\times T \to A^G$. Let $s\in S, t\in T$. Let $t'=\phi(s,t)$ be defined as follows, in the following cases:
  \begin{enumerate}
  
      \item $g = k h$ for some $k\in G$ with $s(k)=1$ and some $h\in V$:
      
      In this case, let $t'(g) = u(h)$.
      
      \item $g = k h$ for some $k\in G$ with $s(k)=1$ and some $h\in VU^2\setminus V$:
      
      In this case let $E_1=kV$ and $E_2 = kVU^2X^2\setminus k VU^2$. Since $E_1 U\cap E_2 U=\emptyset$, there is a $v\in T$ with $v|_{E_1} = (ku)|_{E_1}$ and $v|_{E_2} = t|_{E_2}$. If there are multiple choices for $v$, choose the $v$ such that the restriction of $k^{-1}v$ to $F = VU^2X^2\setminus VU^2$ is lexicographically least for a fixed ordering of $F$ and a fixed ordering of $A$. Let $t'(g) = v(kh)$.
      
      \item $g\neq kh$ for $s(k)=1$ and $h\in VU^2$:
      
      In this case, let $t'(g) = t(g)$.
  \end{enumerate}
  Since the $1$'s in $S$ are $Z$-apart, this leads to a well-defined definition for $t'$. Informally, $t'$ is constructed from $t$ as follows: the configuration $t'$ mostly agrees with $t$. The first exceptions are the $V$-neighborhoods of any $k \in G$ such that $s(k)=1$, where we set $t'$ to equal the pattern that appears around the origin in $u$. The second exceptions are the borders of these $V$-neighborhoods, where some adjustments need to be made so that---as we explain below---$t'$ and $t$ agree on any translate of $X$.  This construction is schematically depicted in Figure~\ref{fig:phi eps-prox approx}.
  
  \begin{figure}[t]
    \includegraphics[scale=0.8,center]{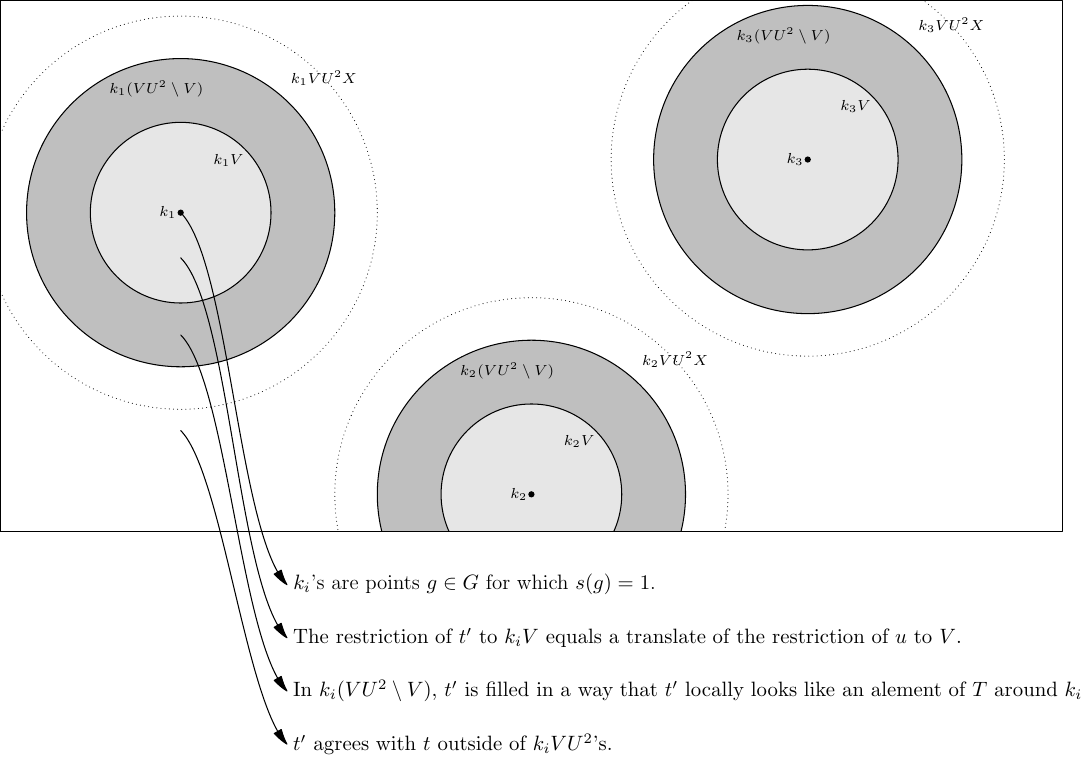}
    \caption{\label{fig:phi eps-prox approx}$t'=\phi(s,t)$ for $s\in S$ and $t\in T$.}
  \end{figure}

  The following hold:
  \begin{itemize}
  
      \item $\phi$ is continuous and equivariant. So $T'=\phi(S\times T)$ is a shift.

      \item Since strong irreducibility is closed under taking products and factors, we see that $T'$ is strongly irreducible.
      
      \item Let $t'_1 = \phi(s_1, t_1), t'_2 = \phi(s_2, t_2) \in T'$. Since $S$ is a $(VU^2X)(VU^2X)^{-1}$-witness shift, there is a $g\in G$ with $s_1(g) = s_2(g) = 1$. So, $t'_1|_{gV}$ and $t'_2|_{gV}$ are both translates of $u|_V$, which means $(g^{-1} t'_1)|_V = (g^{-1} t'_2)|_V$. So from the definition of $V$, we get $d(g^{-1}t'_1, g^{-1}t'_2) < \epsilon$. Hence $T'$ is $\epsilon$-proximal.
      
      \item Now we claim that the set of $X$-patterns of $T'$ and $T$ are equal.
      
      First note that since $u|_V$ has all the $X$-patterns in $T$, and $u|_V$ appears in $T'$, we get that all the $X$-patterns of $T$ appear in $T'$.
      
      Now let $t'=\phi(s,t)\in T'$ and fix an $X$-pattern in $t'$, located at $gX$. If $gX$ does not meet any $k (VU^2)$ for $s(k)=1$, then $t'|_{g X} = t|_{g X}$ and so the pattern appears in $T$. If, on the other hand, $g X$ intersects $k (VU^2)$ for some $k$ with $s(k)=1$ (note that there is at most one such $k$), we have $gX\subseteq k(VU^2X^2)$, and by the definition of $t'$ around $k$, we again see that the pattern in $g X$ appears in $T$.
      
      \item To see $\eps$-minimality of $T'$, let $t'_1,t'_2\in T'$. We know that $t'_2|_X$ is one of the $X$-patterns in $T$, so it appears somewhere in $t'_1$, i.e.\ there exists a $g\in G$ such that $(g t'_1)|_X$ agrees with $t'_2|_X$. We assumed that any two configurations that agree on $X$ have distance less than $\eps$. So, $d(g t'_1, t'_2)<\eps$.

  \end{itemize}
  
This concludes the proof of Proposition \ref{prop:eps-proximal approx}.  

\subsection{Proof of Claim~\ref{clm:G-delta}}
 First we prove proximal shifts are a $G_\delta$. Given a shift $S$, an $\eps>0$ and a $g \in G$, let $P_g \subset S \times S$ be the set of pairs of configurations $s_1,s_2$ such that $d(g s_1,g s_2)<\eps$. Since $P_g$ is the preimage of an open set under a continuous map $P_g$ is open. Thus, whenever $S$ is $\eps$-proximal the collection $\{P_g\,:\, g \in G\}$ forms an open cover of $S \times S$ and thus, by compactness, whenever a shift is $\eps$-proximal there is a finite subset $X \subset G$ which suffice to demonstrate this. For each $X \subset G$, whether  $X$ demonstrates $\eps$-proximality is determined by the restriction of $S$ to a finite set of elements of $G$. But this is exactly the definition of a clopen set in the topology on the space of shifts. Thus the set of $\eps$-proximal shifts is the union of a collection of clopen sets and is therefore open. 
 
 Now we prove minimal shifts are a $G_\delta$. To do this, since minimal shifts are exactly the shifts that are $\eps$-minimal for all $\eps>0$, it is enough to show that the set of $\eps$-minimal shifts is open. Note that by compactness of $X$, a shift is $\eps$-minimal iff its $\eps$-minimality is demonstrated by a finite set. Thus $\eps$-minimality is determined by the set of $Z$-patterns for a finite large enough $Z\subseteq G$. So, as above, the set of $\eps$-minimal shifts is a union of clopen sets, so it is open.

\subsection{Proof of Proposition~\ref{prop:faithful}}
 By Proposition~\ref{prop:dense-G_delta} the minimal proximal shifts are a dense $G_\delta$ in $\mathcal S$. It thus remains to be shown that faithfulness is also generic. Given an element $g\in G$ call a shift {\em $g$-faithful} if $g$ acts non-trivially on the shift. It is easy to see that $g$-faithfulness is an open condition, and so the intersection over all non-trivial $g \in G$, which is faithfulness of the action of $G$, is a $G_\delta$ set. It remains to show that it is dense. To do this we show that each non-trivial $g \in G$ acts non-trivially on every non-trivial strongly irreducible shift. Suppose $g$ is not the identity and acts trivially on a  shift $S$. Then all conjugates of $g$ also act trivially on $S$, so that $hs=s$ for every $h$ a conjugate of $g$ and $s\in S$. In particular, $s(h^{-1})$ must be the same for every $h$ a conjugate of $g$ and every $s \in S$. Since $g$ has an infinite conjugacy class this holds for infinitely many such $h$. But if $S$ is strongly irreducible and non-trivial, then there is some finite $X \subset G$ such that, if $h \notin X$ then there is an $s \in S$ such that $s(h) \neq s(e)$. Thus $g$ must act non-trivially on every non-trivial strongly irreducible shift, and so we have proved the claim.

%\section{Introduction}
\bibliography{main}
\end{document}